\documentclass[preprint,12pt]{elsarticle}
\usepackage{amssymb}
 \usepackage{amsthm}
 \newtheorem{lemma}{Lemma}[section]
 \newtheorem{theorem}{Theorem}[section]
 \newtheorem{remark}{Remark}[section]
\newtheorem{corollary}{Corollary}[section]
\journal{Discrete Mathematics}
\begin{document}
\begin{frontmatter}

\title{On the edge connectivity of direct products with dense graphs}
\author{Wei Wang  }
\author{Zhidan Yan \corref{cor1}}
\cortext[cor1]{Corresponding author.
Tel:+86-997-4680821; Fax:+86-997-4682766.\\
\ead{yanzhidan.math@gmail.com}}

\address{College of Information Engineering, Tarim University, Alar 843300, China}

\begin{abstract}
Let $\kappa'(G)$ be the edge connectivity of $G$ and $G\times H$ the
direct product of $G$ and $H$. Let $H$ be an arbitrary dense graph
with minimal degree $\delta(H)>|H|/2$. We prove that for any graph
$G$, $\kappa'(G\times
H)=\textup{min}\{2\kappa'(G)e(H),\delta(G)\delta(H)\}$, where $e(H)$
denotes the number of edges in $H$. In addition, the structure of
minimum edge cuts is described. As an application, we  present a
necessary and sufficient condition for $G\times K_n(n\ge3)$ to be
super edge connected.

\end{abstract}

\begin{keyword}
direct product  graphs \sep edge connectivity  \sep minimum edge
cuts
\end{keyword}

\end{frontmatter}

\section{Introduction}
\label{intro} All graphs considered in this paper are finite,
undirected, loopless and without multiple edges. Let $G=(V(G),E(G))$
be a nontrivial graph. The \emph{edge connectivity} $\kappa'(G)$ is
the minimum number of edges whose removal disconnects $G$.  A
minimum disconnecting set of edges is necessarily an edge cut and is
also called a minimum edge cut. The \emph{direct product} $G\times
H$ has vertex set $V(G\times H)=V(G)\times V(H)$. Two vertices
$(x,u),(y,v)$ are adjacent when $xy\in E(G)$ and $uv\in E(H)$.

Weichsel observed half a century ago that the direct product of two
nontrivial graphs $G$ and $H$ is connected if and only if both
factors are connected and not both are bipartite graphs
\cite{weichsel1962}. For a long time, this result was the only one
that considered connectivity of direct product graphs. Recently,
Bre\v{s}ar and \v{S}pacapan \cite{bresar2008} obtained an upper
bound and a low bound on the edge connectivity of direct products.
The exact value of edge connectivity of direct products has been
given in special cases. Yang \cite{yang2007,yangappear} determined
the case when one factor is $K_2$. Based on this result, Ou
\cite{ou2011} presented a sufficient condition for $G\times H$ to be
super edge connected(A graph $G$ is  \emph{super edge connected} if
every minimum edge cut is the set of all edges incident with a
vertex in $G$). The explicit formula for edge connectivity of direct
products of two arbitrary graphs has not be found so far. This is
quite opposite to the case of other three products, namely, the
Cartesian product \cite{klavzar2008,xu2006}, the strong product
\cite{bresar2007,yang2008} and the lexicographic product
\cite{yang2007,yangappear}, where explicit formulae have been
obtained in terms of invariants of factor graphs. We mention that
some results on the (vertex) connectivity and super connectivity of
direct products of graphs have been obtained recently, see
\cite{guji2009,guo2010,mamut2008,wang2010,wangappear}.

In this paper, we investigate the case when one factor, say $H$, has
minimum degree $\delta(H)>|H|/2$. Note that this condition implies
$H$ is a connected nonbipartite graph.

\begin{theorem}\label{anydense}
Let $H$ be a graph with  $\delta(H)>|H|/2$. Then for any graph $G$,
$\kappa'(G\times
H)=\textup{min}\{2\kappa'(G)e(H),\delta(G)\delta(H)\}$.
\end{theorem}

\begin{corollary}
$\kappa'(G\times
K_n)=\textup{min}\{n(n-1)\kappa'(G),(n-1)\delta(G)\}$ for $n\ge 3$.
\end{corollary}

Under the same restriction on $H$ in Theorem \ref{anydense}, we also
characterize the structure of all possible minimum edge cuts of
$G\times H$. For $S_0\subseteq E(G)$ we let
$S=\{(x,u)(y,v),(x,v)(y,u):xy\in S_0,uv\in E(H)\}$ and call it
\emph{induced} by $S_0$. Note $|S|=2|S_0|e(H)$ and $G\times
H-S=(G-S_0)\times H$.

\begin{theorem}\label{structure}
Let $H$ be a graph with $\delta(H)>|H|/2$. Let $S$ be a minimum edge
cut of $G\times H$. Then either $S$ is induced by a minimum edge cut
of $G$, or S is the set of all edges incident with a vertex in
$G\times H$, with the exception that $G=K_2$ and
$H=\overline{K_{2l-1}}\vee lK_2$ for some $l$, where $\vee$ is the
natation for join of graphs and $H$ can be obtained from
$K_{2l-1,2l}$ by adding $l$ independent edges on its right part
containing $2l$ vertices.
\end{theorem}

\begin{corollary}
Let $n\ge3$. With the exception that $G=K_2$ and $n=3$, $G\times
K_n$ is super edge connected if and only if $n\kappa'(G)>\delta(G)$.
\end{corollary}

\section{Proof of the main results}
For $x\in V(G)$, followed \cite{bresar2008}, we let
${_xH}=\{(x,u):u\in V(H)\}$ and call it the \emph{H-fiber} with
respect to $x$.

\begin{lemma}\label{quotient} Let $G$ and $H$ be graphs. For
$S\subseteq E(G\times H)$, define a new graph $G^*$ by \\
\textup{(1)}. $V(G^*)=\{{_xH}:x\in V(G)\}$,\\
\textup{(2)}. $E(G^*)=\{{_xH}{_yH}:E_{G\times H-S}({_xH}, {_yH})\neq
\emptyset\}$ , where $E_{G\times H-S}({_xH}, {_yH})$ denotes the
collection of all edges in $G\times H-S$ with one end in ${_xH}$ and
the other in ${_yH}$.

If $G^*$ is disconnected. Then either \\
\textup{(1)}. $|S|>2\kappa'(G)e(H)$, or\\
\textup{(2)}. $|S|=2\kappa'(G)e(H)$ and $S$ is induced by a minimum
edge cut of $G$.
\end{lemma}

 \begin{proof}
 Since $G^*$ is disconnected, the vertices of $G^*$ can be
 partitioned into two nonempty parts, $X^*$ and $Y^*$, such that
 there are no edges joining a vertex in $X^*$ and a vertex in $Y^*$.
 Let $X=\{x:{_xH}\in X^*\}$ and $Y=\{y:{_yH}\in Y^*$\}. Clearly, $(X,Y)$
 is a partition of $V(G)$ and $E(X,Y)$ is an edge cut of $G$.
 Therefore, $e(X,Y)\ge \kappa'(G)$ with equality if and only if
 $E(X,Y)$ is a minimum edge cut. For each edge $xy\in E(X,Y)$ and
 $uv\in E(H)$, both $(x,u)(y,v)$ and $(x,v)(y,u)$ must belong to $S$
 since otherwise ${_xH}{_yH}$ is an edge of $G^*$, contrary to the
 fact that no edges joining a vertex in $X^*$ and a vertex in $Y^*$.
 Let $S'$ be induced by $E(X,Y)$, that is, $S'=\{(x,u)(y,v),(x,v)(y,u):xy\in E(X,Y),uv\in E(H)\}$.
 Then  $S'\subseteq S$ and $|S|\ge2e(X,Y)e(H)$ with equality if and only if $S$ coincides with $S'$.
 The lemma follows directly from the above two
inequalities with  conditions for equalities.
\end{proof}

\begin{lemma} \label{fibercom} Let $G$ and $H$ be nontrivial connected graphs with
$\delta(H)>|H|/2$. Let $S\subseteq E(G\times H)$. Then each H-fiber
is contained in some component of $G\times H$ whenever
\textup{(1)}.$|S|<\delta(G)\delta(H)$. Moreover, with the exception
that $G=K_2$ and $H=\overline{K_{2l-1}}\vee lK_2$ for some $l$, the
same conclusion also holds whenever \textup{(2)}.
$|S|=\delta(G)\delta(H)$ and $S$ is not the collection of all edges
incident with any vertex in $G\times H$.
\end{lemma}

\begin{proof}
Suppose to the contrary that there exists ${_xH}$ that is not
contained in any component of $G\times H-S$. Then there must exist a
component $C$ such that $0<|_xH\cap C|\le|H|/2$. Either of the two
conditions on $S$ implies $|S|\le\delta(G)\delta(H)$ and $G\times
H-S$ has no isolated vertices.

Pick any vertex $(x,u)\in{_xH}\cap C$, we will evaluate the number
of distinct vertices in ${_xH}$ which are linked to $(x,u)$ by paths
of length two. Split $S$ into two subsets, $S_1$ and $S_2$, $S_1$
containing the edges incident with $(x,u)$ and $S_2$ all the others.
Let $A=N_{G\times H-S}(x,u)$, the neighbor set of $(x,u)$ in
$G\times H-S$. We have
\begin{eqnarray}\label{A}
 |A|&=&|N_{G\times H-S}(x,u)|\nonumber\\
    &=&|N_{G\times H}(x,u)|-|S_1|\nonumber\\
    &=&d_G(x)d_H(u)-|S_1|\nonumber\\
    &\ge&\delta(G)\delta(H)-|S_1|\nonumber\\
    &\ge&|S|-|S_1|\nonumber\\
    &=&|S_2|.
\end{eqnarray}
Let $B=E_{G\times H-S}(A, {_xH}\setminus(x,u))$. Then,
\begin{eqnarray}\label{B}
 |B|&=&e_{G\times H-S}(A, {_xH}\setminus(x,u))\nonumber\\
    &\ge&e_{G\times H}(A, {_xH}\setminus(x,u))-|S_2|\nonumber\\
    &=&{\sum_{(y,v)\in A}(d_H(v)-1)}-|S_2|\nonumber\\
    &\ge&|A|(\delta(H)-1)-|S_2|\nonumber\\
    &\ge&|A|(\delta(H)-1)-|A|\nonumber\\
    &=&|A|(\delta(H)-2).
\end{eqnarray}
Let $p$ denote the number of distinct vertices in
${_xH}\setminus(x,u)$ incident with some edges in $B$. Clearly, each
vertex in ${_xH}\setminus(x,u)$ is incident with at most $|A|\ne0$
edges in $B$, and hence,
\begin{equation}\label{p}
p\ge\frac{|B|}{|A|}\ge\frac{|A|(\delta(H)-2)}{|A|}=\delta(H)-2.
\end{equation}
Let $D$ denote the collection of these $p$ vertices which are linked
to  $(x,u)$ by paths of length two, together with $(x,u)$. Since $C$
is the component containing $(x,u)$, it follows that
\begin{equation}\label{D}
|_xH\cap C|\ge |D|=p+1\ge\delta(H)-1.
\end{equation}
We will get a contradiction in either conditions.

(1). $|S|<\delta(G)\delta(H)$. The last inequalities in (\ref{A})
and hence in (\ref{B})-(\ref{D}) will become strict. In particular,
by (\ref{D}), we have $|_xH\cap C|\ge\delta(H)>|H|/2$, a
contradiction.

(2). $|S|=\delta(G)\delta(H)$ and $S$ is not the collection of all
edges incident with any vertex in $G\times H$. We may assume that
equality holds throughout (\ref{A}) and (\ref{B}) since otherwise we
will get the same contradiction as in condition (1).

Let $_yH$ be an $H$-fiber containing a neighbor $(y,v)$ of $(x,u)$
in $G\times H-S$. We claim that no edges in $S$ are incident with
$(y,u)$. Recall $S_1\subseteq S$  contains the edges incident with
$(x,u)$ and $S_2=S\setminus S_1$. Since $(x,u)$ is not adjacent with
$(y,u)$ in $G\times H$, each edge in $S_1$ is not incident with
$(y,u)$. On the other hand, by our assumption,  the first inequality
in (\ref{B}) becomes an equality, which implies

\begin{equation}\label{s2} S_2\subseteq E_{G\times H}(A, {_xH}\setminus(x,u)).\end{equation}
It follows that each edge in $S_2$ is not incident with $(y,u)$
since $(y,u)\notin A$ and $(y,u)\notin {_xH}\setminus(x,u)$. The
claim is verified.

Let $E=N_{G\times H-S}(y,u)\cap{_xH}$. Then, by the claim,
\begin{eqnarray}\label{E}
 |E|&=&|N_{G\times H-S}(y,u)\cap{_xH}|\nonumber\\
    &=&|N_{G\times H}(y,u)\cap{_xH}|\nonumber\\
    &=&|d_H(u)|\nonumber\\
    &\ge&\delta(H).
\end{eqnarray}

\textbf{Case 1}: $|G|\ge3$. We will show that ${_yH}$ is contained
in some component of $G\times H-S$, which implies, in particular,
$(y,u)$ is reachable from $(y,v)$,and hence from its  neighbor
$(x,u)$. It follows that $E\subseteq C$ and  $|{_xH}\cap
C|\ge|E|>|H|/2$ by (\ref{E}), a contradiction.

First, we claim $d_G(y)\ge2$. Suppose $d_G(y)=1$ and hence
$\delta(G)=1$. Then $G-y$ is also connected and nontrivial.
Therefore, $d_G(x)=d_{G-y}(x)+1>\delta(G)$, which implies that the
first inequality in (1) is strict, contrary to our assumption.

Next, pick a neighbor $z$ of $y$, other than $x$. By (\ref{s2}) and
the definition of $S_1$, each edge in $S=S_1\cup S_2$ has an end in
$_xH$ , which implies
\begin{equation}
(G\times H-S)[{_yH}\cup{_zH}]=(G\times H)[{_yH}\cup{_zH}]=K_2\times
H.
\end{equation}

Finally, $K_2\times H$ is connected since $H$ is a connected
nonbipartite graph. This completes the proof for this case.

\textbf{Case 2}: $|G|=2$ and hence $G=K_2$. We will prove
$H=\overline{K_{2l-1}}\vee lK_2$ for some $l$.

If $|H|$ is even. Let $|H|=2k$ and hence $\delta(H)\ge k+1$. Then,
combining (\ref{D}) and (\ref{E}), we have
\begin{equation}
|D|+|E|\ge\delta(H)-1+\delta(H)\ge2k+1>|H|,
\end{equation}
which implies $D\cap E\ne\emptyset$ and hence $E\subseteq C$, a
contradiction as in case 1.

Now we assume $|H|=2k+1$ and hence $\delta(H)\ge k+1$ for some $k$.
The case $k=1$ is trivial since $K_3=\overline{K_1}\vee K_2$ is the
only graph on
 three vertices with $\delta(H)\ge k+1=2$. We assume $k\ge2$.
 We claim that $K_2\times H-S$   has exactly two components.
 First since $G\times H-S$ contains no isolated vertices, each component
 must contain vertices from both ${_xH}$ and ${_yH}$.
Next, by (\ref{D}), any component contains at least $\delta(H)-1\ge
k$ vertices from either of the two fibers. Note $k\ge2$ and the
claim follows. Let $C_1={_xP}\cup{_yS}$ and $C_2={_xQ}\cup{_yT}$ be
the two components of $G\times H-S$, where $(P,Q)$ and $(S,T)$ are
two 'equitable' partitions of $V(H)$. Without loss of generality, we
may assume $|P|=k$ and $|Q|=k+1$. If $|T|=k+1$, then $Q$ and $T$
have a nonempty intersection, which implies $C_2$ is not a complete
bipartite graph and hence $e(C_2)\le(k+1)^2-1.$ Therefore,
\begin{eqnarray}\label{twocom}
 (2k+1)\delta(H)&\le&2e(H)\nonumber\\
                &=&e(K_2\times H-S)+|S|\nonumber\\
                &=&e(C_1)+e(C_2)+\delta(H)\nonumber\\
                &\le&\mbox{max}\{k(k+1)+k(k+1),k^2+(k+1)^2-1\}+\delta(H)\nonumber\\
                &=&2k^2+2k+\delta(H),
\end{eqnarray}
  which implies $\delta(H)\le k+1$ and in fact $\delta(H)=k+1$.
  Consequently, equality holds throughout  (\ref{twocom}). The
  fact that the first inequality becomes an equality means $H$ is
  regular of degree $\delta(H)=k+1$ and $k$ must be odd, while the
  same fact for the second inequality means either both $C_1$ and
  $C_ 2$ are isomorphic to $K_{k,k+1}$, or $C_1=K_{k,k}$ and $C_2$
  is obtained from $K_{k+1,k+1}$by deleting one edge. Note in either
  cases, $C_2$ contains $K_{k,k+1}$ as a subgraph, which implies $H$ contains $K_{k,k+1}$ as a subgraph.
  Let $k=2l-1$. It is easy to check that $H=\overline{K_{2l-1}}\vee
  lK_2$ is the only graph which is both $2l$-regular and contains
  $K_{2l-1,2l}$ as a span subgraph. This completes the proof.
\end{proof}

\begin{remark}
For the exception case $G=K_2$ with vertex set $\{x,y\}$ and
$H=\overline{K_{2l-1}}\vee
  lK_2$,  we let $S_0$ consist  of the $l$ edges of  $lK_2$ in $\overline{K_{2l-1}}\vee
  lK_2$ and $S=\{(x,u)(y,v),(x,v)(y,u):uv\in S_0\}$.
  Then $G\times H-S$ is disconnected since $G\times H-S=K_2\times(H-S_0)=K_2\times
  K_{2l-1,2l}$ is the direct product of two bipartite graphs.
  Note $S$ satisfies condition \textup{(2)} in Lemma
  \textup{\ref{structure}}.
\end{remark}
  \textbf{Proof of Theorem \ref{anydense}}
  We assume $G$ is nontrivial and connected since otherwise $G\times H$ is disconnected and hence the theorem holds.

Clearly, $\kappa'(G\times H)\le\delta(G\times
H)=\delta(G)\delta(H)$.
   Let $S_0$ be a minimum edge cut of $G$, then the induced set
$S=\{(x,u)(y,v),(x,v)(y,u):xy\in S_0,uv\in E(H)\}$ is an edge cut of
$G\times H$ with cardinality $2\kappa'(G)e(H)$. Therefore,
$\kappa'(G\times
H)\le\textup{min}\{2\kappa'(G)e(H),\delta(G)\delta(H)\}$. For the
other inequality, let $S$ be a minimum edge cut of $G\times H$. Then
either $G^*$(defined in Lemma \ref{quotient}) is disconnected, or
there exists an $H$-fiber $_xH$ that is not contained in any
component of $G\times H-S$. If the first result happens, then
$|S|\ge2\kappa'(G)e(H)$ by lemma \ref{quotient}. If the second
result happens, then $|S|\ge\delta(G)\delta(H)$ by the first part of
lemma \ref{fibercom}. Either cases implies $\kappa'(G\times
H)\ge\textup{min}\{2\kappa'(G)e(H),\delta(G)\delta(H)\}$.\qed

\textbf{Proof of Theorem \ref{structure}} Clearly, we may assume $G$
is nontrivial and connected. let $S$ be a minimum edge cut of
$G\times H$. Then by Theorem \ref{anydense}, we have
$|S|=\textup{min}\{2\kappa'(G)e(H),\delta(G)\delta(H)\}$.  If $G^*$
is disconnected, then, by lemma \ref{quotient},
$|S|=2\kappa'(G)e(H)$ and $S$ is induced by a minimum edge cut of
$G$ since the other case conflicts with the formula of $|S|$.

If there exists an $H$-fiber $_xH$ that is not contained in any
component of $G\times H-S$, then by lemma \ref{fibercom} either
$|S|>\delta(G)\delta(H)$, or $|S|=\delta(G)\delta(H)$ and $S$ is the
collection of all edges incident with a vertex in $G\times H$.
Similarly, the case $|S|>\delta(G)\delta(H)$ conflicts with the
formula of $|S|$.\qed

\end{document}